\documentclass[letterpaper,11pt]{amsart}

\pdfoutput=1
\usepackage{tikz, tikz-cd}
\usepackage{bbm}
\usetikzlibrary{arrows}
\usepackage[font=footnotesize,labelfont=bf]{caption}
\usepackage{blkarray}
\usepackage{enumitem}
\usepackage{hyperref}
\usepackage{latexsym,array,delarray,epsfig,setspace,cleveref, mathtools,amssymb,mathrsfs,bbold}
\usepackage{subfig}

\newtheorem{theorem}{Theorem}
\numberwithin{theorem}{section}
\newtheorem{proposition}[theorem]{Proposition}
\newtheorem{corollary}[theorem]{Corollary}
\newtheorem{lemma}[theorem]{Lemma}

\theoremstyle{definition}
\newtheorem{definition}[theorem]{Definition}
\newtheorem{remark}[theorem]{Remark}
\newtheorem{example}[theorem]{Example}

\newcommand{\Z}{{\mathbb Z}}

\renewcommand{\k}{\mathbb{k}}

\newcommand{\R}{\mathbb{R}}

\newcommand{\B}{\mathcal{B}}

\newcommand{\F}{\mathcal{F}}

\newcommand{\M}{\mathcal{M}}
\newcommand{\E}{\mathcal{E}}

\newcommand{\G}{\mathcal{G}}

\renewcommand{\P}{{\mathbb P}}

\newcommand{\bv}{{\bf v}}
\newcommand{\e}{\textbf{e}}

\newcommand{\bt}{\textbf{t}}

\DeclareMathOperator{\GL}{GL}

\DeclareMathOperator{\rank}{rank}
\DeclareMathOperator{\PL}{PL}
\DeclareMathOperator{\Berg}{Berg}
\DeclareMathOperator{\Span}{span}
\DeclareMathOperator{\conv}{conv}
\DeclareMathOperator{\cone}{cone}

\DeclareMathOperator{\codim}{codim}
\DeclareMathOperator{\Int}{Int}
\DeclareMathOperator{\relint}{relint}
\DeclareMathOperator{\Todd}{Todd}

\newcommand{\h}{\mathfrak{h}}

\newcommand{\GF}{\textup{GF}}

\newcommand{\LL}{\mathcal{L}}

\begin{document}

\title{On Ehrhart theory for tropical vector bundles}

\author{Suhyon Chong}
\address{Department of Mathematics, University of Pittsburgh,
Pittsburgh, PA, USA.}
\email{suc86@pitt.edu}

\author{Kiumars Kaveh}
\address{Department of Mathematics, University of Pittsburgh, Pittsburgh, PA, USA.}
\email{kaveh@pitt.edu}

\date{\today}

\begin{abstract}
The notion of a tropical vector bundle on a toric variety was recently introduced in \cite{Khan-Maclagan} and \cite{KM-trop-vb}. In this paper, we study the Euler characteristic and rank of global sections for tropical vector bundles. 
We associate a convex chain (a finite integer linear combination of indicator functions of convex polytopes) to a tropical vector bundle encoding its Euler characteristic. We then see that the Khovanskii-Pukhlikov theory of convex chains gives a combinatorial Hirzebruch-Riemann-Roch theorem for tropical vector bundles. This, in particular, applies to toric vector bundles. Also, we extend Klyachko's resolution of a toric vector bundle by split toric vector bundles to tropical vector bundles. As shown in \cite{KM-trop-vb}, every matroid comes with a tautological tropical vector bundle. We answer positively a question posed in \cite{KM-trop-vb} about equality of Euler characteristic with rank of space of global sections (in other words, vanishing of higher cohomologies) for the tautological bundle of a matroid. 
\end{abstract}

\maketitle
\tableofcontents







\section{Introduction}
The theory of divisors and line bundles on tropical varieties is a very well developed subject. On the contrary, the theory of higher rank vector bundles on tropical varieties has been lacking a general framework and is not developed as much. An early work on tropical vector bundles is \cite{Allermann}. More recently, tropical vector bundles on metric graphs have been introduced in \cite{Ulirsch}. Also, a notion of a tropical vector bundle has been explored by Jun, Mincheva and Tolliver (\cite{Kalina}) from the point of view of tropical scheme theory (in their approach, vector bundles are always direct sums of line bundles).

Recently, in \cite{KM-trop-vb, Khan-Maclagan} (and almost exactly at the same time) the notion of a \emph{tropical vector bundle} over a toric variety has been introduced. When the matroid data in the definition is representable, these are tropicalizations of toric vector bundles (torus equivariant vector bundles on toric varieties). 

Toric vector bundles of rank $r$ were famously classified by Klyachko (in the remarkable paper \cite{Klyachko}) in terms of compatible systems of filtrations on an $r$-dimensional vector space (the first classification of toric vector bundles goes back to Kaneyama \cite{Kaneyama}). 
Klyachko's classification has influenced much of work in toric vector bundles, and in particular, serves as the motivation for the definition of a tropical vector bundle in  \cite{Khan-Maclagan, KM-trop-vb}.


A tropical vector bundle on a toric variety, in the sense of \cite{Khan-Maclagan}, is a certain compatible collection of filtrations by flats of a matroid. On the other hand, a tropical vector bundle on a toric variety, in the sense of \cite{KM-trop-vb}, is a piecewise linear map from the fan to the (lifted) Bergman fan of a matroid (see Section \ref{subsec-trop-vb}). One can show that these are equivalent sets of data. In this paper, we follow the approach in \cite{KM-trop-vb}.

The goal of the present paper is to develop an Ehrhart theory for tropical vector bundles. Our main tool is Khovanskii-Pukhlikov theory of \emph{convex chains} (\cite{Khovanskii-Pukhlikov-1, Khovanskii-Pukhlikov-2}). We show that a tropical vector bundle on a complete toric variety naturally determines a convex chain whose values on lattice points gives the equivariant Euler characteristic of the bundle. This in turn yields a combinatorial Hirzebruch–Riemann–Roch type formula for tropical vector bundles.

In \cite{Khovanskii-Pukhlikov-1}, Khovanskii-Pukhlikov extend the Ehrhart theory of lattice convex polytopes to lattice convex chains. A \emph{convex chain} is a function $\alpha:\R^n \to \Z$ of the form $\alpha=\sum_i n_i\mathbbm{1}_{P_i}$, where $n_i\in\Z$ and $P_i$ are convex polytopes in $\R^n$. Here $\mathbbm{1}_P$ denotes the indicator function of a polytope $P$. When the $P_i$ are lattice polytopes, we call $\alpha$ a \emph{lattice convex chain}. For a convex chain $\alpha$, one considers $S(\alpha)$, the sum of values of $\alpha$ on lattice points, and $I(\alpha)$, the integral of $\alpha$. These are generalizations of the number of lattice points and volume of a polytope respectively. The algebra of convex chains is closely related (but not identical) to McMullen's polytope algebra (\cite{McMullen}). See \cite[Appendix]{EHL} for a good overview of the comparison between McMullen's polytope algebra and the convex chains perspective. 

One of the main results in \cite{Khovanskii-Pukhlikov-1} is that the function $\alpha \mapsto S(\alpha)$ behaves in a polynomial fashion for lattice convex chains. This far extends the Ehrhart polynomial of a lattice polytope. Furthermore, in \cite{Khovanskii-Pukhlikov-2}, they give a Hirzebruch-Riemann-Roch formula relating the sum $S(\alpha)$ and the integral $I(\alpha)$.  

Let $\E$ be a tropical vector bundle. In Section \ref{subsec-cc-trop-vb}, we show how to assign a convex chain $\alpha_\E$ to $\E$ (using equivariant Chern roots of $\E$). Our first main result is that the convex chain $\alpha_\E$ exactly computes the equivariant Euler characteristic of $\E$ (see Theorem \ref{Euler-Ch}):

\begin{theorem}\label{thm-Euler char-convex chain}
Let $\E$ be a tropical vector bundle on a toric variety $X_\Sigma$ with fan $\Sigma$, and let $\alpha_{\E}$ be the associated convex chain. Then for every character $u$ of the torus, $$\chi(X_\Sigma,\E)_u = \alpha_{\E}(u).$$  
\end{theorem}

A key step in the proof is the invariance of the equivariant Euler characteristic under refinements of the fan (see Theorem \ref{th-toric-pull-back}). This allows to pass to a refinement where the (multi-valued) support function of the tropical vector bundle becomes linear on all the cones in the fan. This then enables us to reduce the claim to the case of a sum of line bundles.

\begin{theorem}\label{thm-invariance}
Let $\Sigma'$ be a refinement of a complete fan $\Sigma$, and let $\E'$ be the pull–back of a tropical vector bundle $\E$ on $X_\Sigma$ to $X_{\Sigma'}$. Then for every character $u\in M$,
$$\chi(X_\Sigma, \E)_u = \chi(X_{\Sigma'}, \E')_u.$$
\end{theorem}

Combining Theorem \ref{thm-Euler char-convex chain} with the Khovanskii-Pukhlikov Hirzebruch–Riemann–Roch theorem for convex chains, immediately yields a combinatorial Hirzebruch-Riemann–Roch formula for tropical vector bundles.
For
$z=(z_\rho)_{\rho\in\Sigma(1)} \in \mathbb{R}^{\Sigma(1)}$, let $P(z)$ be the virtual polytope with the support numbers $z_\rho$ (see Section \ref{sec-KhPukh}, paragraph after Theorem \ref{th-convex-chain-supp-function}).
For the convex chain $\alpha_\E$ associated to a tropical vector bundle $\E$ on $X_\Sigma$, define $\alpha_\E[z]:= \alpha_\E * P(z)$ (see Proposition \ref{prop-convoltuion}).

\begin{corollary}
    Let $\E$ be a tropical vector bundle over $X_\Sigma$ and $\alpha_\E$ be the associated convex chain. Then,
    $$\Todd\left(\frac{\partial}{\partial z}\right) I(\alpha_\E[z]) \Big\vert_{z=0} =\chi(X_\Sigma,\E).$$
\end{corollary}


In \cite{Klyachko}, for a toric vector bundle $\E$ on a smooth complete toric variety $X_\Sigma$, Klyachko constructs a resolution $0 \to \E \to \F_0 \to \cdots \to \F_n \to 0$ by split toric vector bundles on $X_\Sigma$. We recall that a toric vector bundle is \emph{split} if it is equivariantly isomorphic to a direct sum of toric line bundles. In Section \ref{subsec-split-res}, we extend the construction of a split resolution to tropical vector bundles. We note that we do not yet have a notion of morphism of tropical vector bundles so we cannot talk about a resolution. Instead, given a tropical vector bundle $\E$ on a smooth complete toric variety, we construct a sequence of split tropical vector bundles $\F_0, \ldots, \F_n$ such that alternating sum of their equivariant $K$-classes gives the equivariant $K$-class of $\E$ (Theorem \ref{th-K-class-resolution}). This then gives an alternative way to construct the convex chain $\alpha_\E$ which encodes the equivariant Chern roots of $\E$. 

An important example of a tropical vector bundle with a matroid $\M$ is the tautological bundle introduced in \cite{KM-trop-vb}. Namely, any matroid $\M$ comes with a \emph{tautological tropical vector bundle} $\E_\M$. This construction is motivated by two toric vector bundles - the universal subbundle and the universal quotient bundle - arising from the data of a representable matroid in \cite{BEST}. In Section \ref{sec-global sec-Euler c-taut}, we describe the global sections matroid and the Euler characteristic of this bundle and show that the rank of global sections is equal to the Euler characteristic. We interpret this as vanishing of higher cohomologies of the tautological bundle (See Theorem \ref{thm-vasnishing of higher coh-taut}). This is in agreement with the vanishing of higher cohomologies of tautological bundles for the representable case in \cite{Eur}.
\begin{theorem}[Vanishing of higher cohomologies of a tautological bundle]\label{thm-vasnishing-taut}
For any character $u$ we have:
    $$\chi(X_{m},\E_\M)_u= h^0(X_{m},\E_\M)_u.$$
\end{theorem}
\begin{remark}
Although Theorem \ref{thm-vasnishing-taut} suggests vanishing of higher cohomologies, as of now, we do not have a notion of higher cohomologies for tropical vector bundles. We also do not have notions of tensor product, exterior powers, or symmetric powers of tropical vector bundles.
\end{remark}


\subsection*{Acknowledgment}
We would like to thank Askold Khovanskii for suggesting the association of a multi-valued support function and hence a convex chain to a toric vector bundle. We are in debt to Chris Eur and Chris Manon for many helpful discussions. The second author is partially supported by the National Science Foundation Grant DMS-210184 and a Simons Collaboration Grant.
\bigskip

\noindent{\bf Notation:}
\begin{itemize}
    \item $\k$ denotes the ground field.
    \item $T \cong \mathbb{G}_m^n$ denotes an algebraic torus with $M$ and $N$ its character and cocharacter lattices respectively. We denote the pairing between them by $\langle \cdot, \cdot \rangle: M \times N \to \Z$. We let $M_\R = M \otimes \R$ and $N_\R = N \otimes \R$ be the corresponding $\R$-vector spaces. 

    \item $U_\sigma$ is the affine toric variety corresponding to a (strictly convex rational polyhedral) cone $\sigma \subset N_\R$.
    \item $\Sigma$ is a fan in $N_\R$ with corresponding toric variety $X_\Sigma$. We denote the support of $\Sigma$ by $|\Sigma|$. 
    \item $\PL(N_\R, \R)$ is the set of piecewise linear functions on $N_\R$. We denote the subset of piecewise linear functions that attain integer values on $N$ by $\PL(N, \Z)$. Finally $\PL(\Sigma, \R)$ (respectively $\PL(\Sigma, \Z)$) denotes the subset of piecewise linear functions (respectively integral piecewise linear functions) that are linear on cones in $\Sigma$.
    \item $\M$ denotes a matroid with ground set $\G$.
    \item $ \widetilde{\Berg}(\M)$  denotes the (lifted) Bergman fan of $\M$.  
    \item $\widetilde{\GF}(\M)$ denotes the (lifted) Gr\"obner fan of $\M$.
    \item $\E$ denotes a rank $r$ tropical vector bundle on $X_\Sigma$ with piecewise linear map $\Phi_\E: |\Sigma| \to \widetilde{\Berg}(\M)$.
    \item  $(F^\rho_i)$ is a Klyachko filtration by flats associated to a ray $\rho$. 
    \item $\h_\E$ is the multi-valued support function of $\E$ given by equivariant Chern roots.
    \item $\alpha_\E$ is the convex chain associated to $\E$.
    \item $\E_\M$ denotes tautological bundle of matroid $\M$.
    
\end{itemize}

\section{Preliminaries}



In this section, we review some background material on toric vector bundles and tropical linear spaces. Then we review two equivalent definitions of a tropical vector bundle in \cite{KM-trop-vb} motivated by the classification of toric vector bundles as piecewise linear maps introduced in \cite{KM-TVBs-valuations}. 

Throughout, $T$ denotes an $n$-dimensional torus over a field $\k$. We denote the character and cocharacter lattices of $T$ by $M$ and $N$ respectively.
Also, $X_\Sigma$ denotes a complete $T$-toric variety with fan $\Sigma$ in $N_\R = N \otimes \R$.
\subsection{Klyachko classification of toric vector bundles} 
Let $\E$ be a toric vector bundle over $X_\Sigma$, that is, a vector bundle over $X_\Sigma$ equipped with a $T$-linearization.

For each cone $\sigma\in \Sigma$,  the space of sections $\Gamma(U_\sigma, \mathcal{E}\vert_{U_\sigma})$ is a $T$-module. We let $\Gamma(U_\sigma,\mathcal{E}\vert_{U_\sigma})_u $ be the weight space corresponding to a character $u\in M$.  One has the weight space decomposition:
$$\Gamma(U_\sigma,\mathcal{E}\vert_{U_\sigma}) = \bigoplus_{u\in M } \Gamma(U_\sigma,\mathcal{E}\vert_{U_\sigma})_u$$
We fix a point $x_0$ in the dense orbit $X_0 \subset X_\Sigma$. Let $E = \mathcal{E}_{x_0}$ denote the fiber of $\mathcal{E}$ over $x_0$. Then every section in $\Gamma(U_\sigma,\mathcal{E}\vert_{U_\sigma})_u$ is determined by its value at $x_0$. Thus, by restricting the section to $E$, we obtain an embedding  $\Gamma(U_\sigma,\mathcal{E}\vert_{U_\sigma})_u \xhookrightarrow{} E$.\\

We denote the image of $\Gamma(U_\sigma,\mathcal{E}\vert_{U_\sigma})_u$ in $E$ by $E^\sigma_u$. For $u'\in \sigma^\vee \cap M$, multiplication by $\chi^{u}$ given an injection $\Gamma(U_\sigma,\mathcal{E}\vert_{U_\sigma})_u \xhookrightarrow{} \Gamma(U_\sigma,\mathcal{E}\vert_{U_\sigma})_{u-u'}$ which induces an inclusion $E^\sigma_u \subseteq E^\sigma_{u-u'}$.\\

If $u'\in \sigma^\perp$, then these maps are isomorphisms and $E^\sigma_u$ only depends on $[u] \in M_\sigma = M/(\sigma^\perp \cap M)$. Thus, for $\rho\in \Sigma(1)$,  $E^\rho_i := E_u^\rho$ where $\langle u, \bv_\rho\rangle = i$ is well-defined. Then, we have a decreasing filtration of $E$
$$\cdots\supseteq E^\rho_{i-1} \supseteq E^\rho_i \supseteq E^\rho_{i+1}\supseteq \cdots$$

\begin{proposition}
        Every toric vector bundle $\mathcal{E}$ of rank $r$ on $U_\sigma$ splits equivariantly into a sum of trivial toric line bundles:
    $$\mathcal{E} = \bigoplus_{i=1}^r \mathcal{L}_{u_i}$$
    where $[u_i] \in M_\sigma$. Here $\mathcal{L}_{[u]}$ denotes the trivial line bundle where $T$ acts via $u$.
\end{proposition}
We denote the multiset $\{[u_1],\ldots,[u_r]\}$ by $u(\sigma)$. Then, for $\sigma\in \Sigma$,
the filtrations $(E^\rho_i)_{i\in\mathbb{Z}}, \rho\in \Sigma(1)$ satisfy the following compatibility condition: There is a decomposition of $E$ into a direct sum of 1-dimension subspaces 
$$E = \bigoplus_{[u]\in u(\sigma)} L_{[u]}$$
such that for $\rho\in \sigma(1)$,
$E_i^\rho = \displaystyle\sum_{\langle u,\bv_\rho\rangle \geq i} L_{[u]}$.

The Klyachko data of a rank $r$ toric vector bundle $\mathcal{E}$ on $X_\Sigma$ is a family of decreasing filtrations $(E_i^\rho)$, in an $r$-dimensional vector space $E$, satisfying the compatibility condition:
 \begin{itemize}
     \item For each $\sigma\in\Sigma$, there is a basis $B_\sigma =\{b_{\sigma,1},\ldots,b_{\sigma,r}\}$ in $E$ and characters $u(\sigma)=\{u_{\sigma,1},\ldots,u_{\sigma,r}\}\subset M$ such that $E_i^\rho = \Span\{b_{\sigma,j}\in B_\sigma \mid \langle u_{\sigma,j},\bv_\rho\rangle \geq i\}$ for each $\rho\in \sigma(1)$.
\end{itemize}

\begin{theorem}[Klyachko]
The category of toric vector bundles on $X_\Sigma$ is equivalent to the
 category of compatible filtrations on finite dimensional $\k$-vector spaces.    
\end{theorem}

Let $\widetilde{\mathcal{B}}(E)$ be the space of decreasing $\mathbb{R}$-filtrations in $E$. For a basis $B \subset E$, let $\widetilde{A}(B) \simeq \mathbb{R}^r$ be the space of decreasing $\mathbb{R}$- filtrations spanned by subsets of $B$.

\begin{definition}
        A \emph{piecewise linear map} $\Phi: |\Sigma|\to \widetilde{\mathcal{B}}(E)$ is a map satisfying:
    \begin{itemize}
        \item For any $\sigma\in \Sigma$, there is a basis $B$ of $E$ such that $\Phi(\sigma)\subseteq \widetilde{{A}}(B)$
        \item $\Phi\,|_{\sigma}: \sigma \to \widetilde{A}(B)$ is linear.
    \end{itemize}
\end{definition}

The following is from \cite{KM-TVBs-valuations}:
\begin{theorem}
    The category of toric vector bundles on $X_\Sigma$ is naturally equivalent to the category of integral
piecewise linear maps to $\widetilde{\mathcal{B}}(E)$, for all finite dimensional $\mathbf{k}$-vector spaces $E$.
\end{theorem}
\begin{remark}
    The classification of toric vector bundles as piecewise linear maps is a reformulation of Klyachko data of toric vector bundles.
\end{remark}

\subsection{Matroids}
Throughout $\M$ denotes a matroid of rank $r$ with ground set $\G = \{1,\ldots,m\}$. 
Recall that $\{\e_1, \ldots, \e_m\}$ denotes the standard basis for $\R^\G = \R^m$ and for a subset $S \subseteq \G$ we put $\e_S = \sum_{j \in S} \e_j$. To the matroid $\M$, one associates two polyhedral fans as we recall below.

The \emph{matroid polytope} $P_\M$ is the convex hull of $\{\e_B \mid B \text{ is a basis in $\M$}\}$. One shows that every $\e_B$ is a vertex of $P_\M$. The polytope $P_\M$ lies in the hyperplane given by the coordinate sum equals to $r$ and has dimension $m-1$.

Recall that the \emph{outer normal fan} of a polytope $P \subset \R^m$ is a complete fan in $\R^m/L(P)^\perp$ where $L(P) \subset \R^m$ is the affine span of $P$.

\begin{definition}   
 The \emph{Gr\"obner fan} $\GF(\M) \subset \R^m/\R\e_{[m]}$ is the outer normal fan of the matroid polytope $P_\M$.

 The \emph{lifted Gr\"obner fan} $\widetilde{\GF}(\M)\subset \R^m$ is the preimage of $GF(\M)$ under the quotient map $\R^m\to \R^m/\R\e_{[m]}$.
\end{definition}

 By definition, the cones in the Gr\"obner fan $\GF(\M)$ are in one-to-one correspondence with the faces of $P_\M$. In particular, maximal cones in $\GF(\M)$ correspond to bases of $\M$.  Let $\sigma_F$ denote the face of $\GF(\M)$ corresponding to a face $F$ of the matroid polytope $P_\M$ and $\widetilde{\sigma}_F$ denote the corresponding cone in $\widetilde{\GF}(\M)$.  One shows that the bases of $\M$ corresponding to the vertices of $F$ define a matroid $\M_F$ (on the ground set $\M$) called the \emph{initial matroid} of $\M$ associated to $F$. 

The bases of $\M$ corresponding to the vertices of a face $F$ define a matroid called the \emph{initial matroid} associated to $F$.

For $i\in E$, let $\pi_i:\R^m \to \R$ be the projection onto the $i$-th coordinate. For a circuit $C$ and $w\in R^m$ let $\pi_C(w)=\min\{w_j\mid j \in C\}.$

Let $\sigma \in \widetilde{\GF}(\M)$ and take $w$
 in the relative interior of $\sigma$. Then for any circuit $C$, there are \emph{winner} coordinates in $w$, that is, $i\in C$ such that $w_i = \pi_C(w).$
The collection of winners for all possible circuits uniquely determines a cone $\sigma \in \widetilde{\GF}(\M)$ and the same relationship holds for $\GF(\M)$.

Next we recall the Bergman of $\M$. Let $\F=(F_1 \subsetneqq \cdots \subsetneqq F_k = \M)$ be a flag of flats of $\M$. We define the convex polyhedral cone $\widetilde{\sigma}_\F$ by: $$\widetilde{\sigma}_\F = \cone\{ \e_F \mid F \in \F \}\subset \R^m.$$
Let $\sigma_\F$ denote image of $\widetilde{\sigma}_\F$ under the projection $\R^m \to \R^m/\R\e_{[m]}.$
\begin{definition}
    The \emph{Bergman fan} $\Berg(\M)$ is the fan in $\R^m/\R\e_{[m]}
    $ consisting of the cones $\sigma_\F$ for all the flags of flats in $\M$. The \emph{lifted} Bergman fan $\widetilde{\Berg}(\M)$ is the fan in $\R^m$ consisting of the cones $\widetilde{\sigma}_\F$.
\end{definition}

A relationship between the two fans is the following

\begin{proposition} 
The support of $\widetilde{\Berg}(\M)$ is the support of a subfan of $\widetilde{\GF}(\M)$. The lifted Bergman fan consists of cones $\widetilde{\sigma} \in \GF(\M)$ such that  the initial matroid for corresponding faces $F$ in $P_\M$ are loop-free.
\end{proposition}

In analogy with the classification of toric vector bundles in \cite{KM-TVBs-valuations}, we call a subset of the Bergman fan obtained by intersecting it with a maximal cone in $\GF(\M)$ an apartment.

\begin{definition}[Apartment in Bergman fan] \label{def-apartment-Bergman-fan}
Let $B \subseteq \G$ be a basis, and let $\widetilde{\sigma}_B$ be the corresponding maximal cone in $\widetilde{\GF}(\M)$. We define the \emph{apartment} $\widetilde{A}_B$ to be the intersection $$\widetilde{A}_B = \widetilde{\Berg}(\M) \cap \widetilde{\sigma}_B.$$
We also let $A_B$ be the image of $\widetilde{A}_B$ in $\Berg(\M)$
\end{definition}

\begin{remark}
    The notion of apartment is not new and has first introduced by Felipe Rinc\'{o}n under the name local tropical linear space (\cite{Rincon}). 
\end{remark}
The following can be found in \cite{KM-trop-vb}.
\begin{proposition}\label{prop-apartment-matroid}
Let $B$ be a basis in $\M$ with corresponding apartment $A_B$. We have the following:
\begin{itemize}
\item[(a)] $\widetilde{A}_B$ is a union of cones in the lifted Bergman fan.
\item[(b)] $\widetilde{A}_B$ is piecewise linearly isomorphic to $\R^r$.
\item[(c)] As a simplicial complex, $\widetilde{A}_B$ is isomorphic to the Coxeter complex of type $A_{r-1}$ where $r=\rank(\M)$. 
\end{itemize}
\end{proposition}

\subsection{Tropical vector bundles}  \label{subsec-trop-vb}
Let $\Sigma$ be a fan in $N_\R$ and let $\M$ be a matroid. For simplicity we assume $\Sigma$ is a complete fan. Following \cite{KM-trop-vb}, we have two equivalent definitions of a tropical vector bundle over a toric variety $X_\Sigma$.
\begin{definition}[Tropical vector bundle by Klyachko data] \label{def-matroid-vb-1}
A \emph{tropical vector bundle} $\E$ over the toric variety $X_\Sigma$ is the data of a family of decreasing $\Z$-filtrations of flats in a matroid $\mathcal{M}$ of rank $r$, $(F_i^\rho)_{i\in\Z}$ for each $\rho \in \Sigma(1)$ that satisfy the compatibility condition,
 \begin{itemize}
     \item For each $\sigma\in\Sigma$, there is a basis $B_\sigma =\{b_{\sigma,1},\ldots,b_{\sigma, r}\}$ in $\mathcal{M}$ and characters $u(\sigma)=\{u_{\sigma,1},\ldots,u_{\sigma,r}\}\subset M$ such that $F_i^\rho = \Span\{b_{\sigma,j} \,\vert\, \langle u_{\sigma,j},\bv_\rho\rangle \geq i\}$ for each $\rho\in \sigma(1)$.
 \end{itemize}
\end{definition}

 This is an analogue of the Klyachko filtrations for toric vector bundles. 

\begin{definition}[Tropical vector bundle as a piecewise linear map]  \label{def-matroid-vb-2}
A \emph{tropical vector bundle} $\E$ on $X_\Sigma$ is the data of a  map $\Phi: |\Sigma| \to \widetilde{\text{Berg}}(\mathcal{M})$ with the following property: 
\begin{itemize}
    \item For any cone $\sigma\in\Sigma$, there is a basis $B_\sigma$ in $\mathcal{M}$ such that $\Phi(\sigma)\subseteq \widetilde{A}_{B_\sigma}$

    \item $\Phi\,|_\sigma : \sigma \to \widetilde{A}_{B_\sigma}$ is linear.
\end{itemize}
\end{definition}

A point $w\in \Berg(\widetilde M)$ defines a decreasing filtration $(F_i^{(w)})_{i \in \Z}$ by:
$$F^{(w)}_i := \Span\{ e \in \G \mid w_e \geq i \}, \qquad i \in \Z.$$
Given a piecewise linear map $\Phi:|\Sigma|\to \widetilde{\Berg}(M)$, we apply this construction to each ray generator
$\bv_\rho$ and set
$F^\rho_i := F^{(\Phi(v_\rho))}_i.$
For every cone $\sigma$ there is a basis $B_\sigma$ with $\Phi(\sigma)\subset \widetilde{A}_{B_\sigma}$ and moreover $\Phi|_\sigma$ is linear.
From this it follows that the filtrations $(F^\rho_i)$, $\rho \in \sigma(1)$, satisfy Klyachko's compatibility condition in Definition \ref{def-matroid-vb-1}.

\begin{proposition}
Definitions \ref{def-matroid-vb-1} and \ref{def-matroid-vb-2} are equivalent.   
\end{proposition}

We will refer to either of the data in Definitions \ref{def-matroid-vb-1} or \ref{def-matroid-vb-2} as a \emph{tropical toric vector bundle} $\E$ with  bases $B_\sigma$ with compatible filtrations $(F_i^\rho)_{i\in \Z}$ for $\rho \in \sigma(1)$ and \emph{piecewise linear map} $\Phi$. 

Let $s = |\Sigma(1)|$ be the number of rays of $\Sigma$, then for $t_1, \ldots, t_s \in \Z$, the corresponding \emph{Klyachko flat} is the intersection $F^{\rho_1}_{t_1}\cap \cdots \cap F^{\rho_s}_{t_s} \subseteq \G$.

\begin{definition}[Diagram of a tropical toric vector bundle]  \label{def-Klyachko-flats}
We let $D_\Phi$ be the $s \times m$ integral matrix with rows $\Phi(\bv_\rho) \in \widetilde{\Berg}(\M)$ for $\rho \in \Sigma(1)$, where $m = |\G|$. We call the matrix $D_\Phi$ the \emph{diagram} of the tropical toric vector bundle determined by $\Phi$.
\end{definition}

\begin{proposition}
    If $\Phi$, $\Phi': |\Sigma| \to \widetilde{\Berg}(\M)$ are piecewise linear maps with $D_\Phi = D_{\Phi'}$, then $\Phi = \Phi'$. Let $\Sigma$ be a simplicial fan. If $D$ is any $s\times m$ integral matrix with rows in $\widetilde{\Berg}(\M)$ such that any rows corresponding to rays of a face $\sigma \in \Sigma$ lie in a common apartment of $\widetilde{\Berg}(\M)$, then $D$ determines a piecewise linear map $\Phi_D: |\Sigma| \to \widetilde{\Berg}(\M)$ satisfying Definition \ref{def-matroid-vb-2}. 
\end{proposition}

\begin{definition}[Parliament of polytopes] \label{def-parliament}
    Let $\Phi: |\Sigma| \to \widetilde{\Berg}(\M)$ define a tropical toric vector bundle $\E$ with diagram $D$. 
    For each $e \in \G$ we let $P_{e} \subset M_\R$ be the Newton polytope of the divisor on $X_\Sigma$ defined by the $e$-th column of $D$: 
    \begin{equation} \label{equ-P-v(e)}
P_e = \{ y \in M_\R \mid \langle y, \bv_\rho \rangle \leq \pi_e(\Phi(\bv_\rho)),~\forall \rho \in \Sigma(1) \}.
\end{equation}
The \emph{parliament} of $\E$ is defined to be the collection of polyhedra $\{P_{e} \mid e \in \G\}$.
\end{definition}

\subsection{Sheaf of sections}

In this section, we review the sheaf of sections of a tropical vector bundle (\cite[Section 6]{KM-trop-vb}).
The sheaf of sections associates to each affine toric chart $U_\sigma$ and a character $u \in M$ a matroid $H^0(U_\sigma,\E|_{U_\sigma})_u$.

First we recall the case of toric vector bundles. Let $\E$ be a toric vector bundle over a complete toric variety $X_\Sigma$. For a character $u \in M$, we let  $H^0(X_\Sigma, \E)_u$ denote the $u$-weight space in the space of global sections $H^0(X_\Sigma, \E)$. Similarly, for any cone $\sigma \in \Sigma$, $H^0(U_\sigma, \E)_u$ denotes the $u$-weight space in the space of sections of $\E$ on the affine toric chart $U_\sigma$.

Following \cite{Klyachko}, we express these spaces in terms of the Klyachko spaces $E^\rho_j$ of $\E$. Let $\sigma$ be a maximal cone in $\Sigma$ and let $\sigma(1) = \{\rho_1, \ldots, \rho_s\}$, and let $\bv_i\in N$ be the ray generator of the ray $\rho_i$, then:

$$H^0(U_\sigma, \E\!\!\mid_{U_\sigma}) = \bigoplus_{\bt \in \Z^s} E^{\rho_1}_{t_1} \cap \cdots \cap E^{\rho_s}_{t_s}.$$ 

For a cone $\sigma \in \Sigma$ and $e \in E$ we can associate 
a polyhedron $P_{e, \sigma} \subset M_\R$ defined by:
$$
P_{e, \sigma} = \{ y \in M_\R \mid \langle y, \bv_\rho \rangle \leq \pi_e(\Phi(\bv_\rho)),~\forall \rho \in \sigma(1)\}.
$$
It follows from the definition that for any character $u \in M$ we have:
$$\dim H^0(U_\sigma, \E\!\!\mid_{U_\sigma})_u = \dim \{ e \in E \mid u \in P_{e, \sigma} \}$$
where $E$ here is the fiber over the torus fixed point. 

Similarly, for any $e \in E$, we define the polyhedron from the parliament of $\E$:
$$P_{e} = \{ y \in M_\R \mid \langle y, \bv_\rho \rangle \leq \pi_e(\Phi(\bv_\rho)),~\forall \rho \in \Sigma(1) \}.
$$
We also have $H^0(X_\Sigma, \E)_u = \bigcap_{\rho \in \Sigma(1)} E^\rho_{\langle u,\bv_\rho\rangle}$, so that:
$$\dim H^0(X_\Sigma, \E)_u = \dim(\bigcap_{\rho \in \Sigma(1)} E^\rho_{\langle u, \bv_\rho\rangle}) = \dim \{ e \in E \mid u \in P_{e} \}.$$

Now for a tropical vector bundle over a complete toric variety, the $u$-weight space of global sections $H^0(U_\sigma,\E|_{U_\sigma})_u$ and $H^0(X_\Sigma,\E)_u  $ are defined in analogy with toric vector bundle case. 

\begin{definition}[Sheaf of sections of a tropical vector bundle] \label{def-sheaf of section}
$$H^0(U_\sigma, \E\!\!\mid_{U_\sigma})_u = \bigcap_{\rho\in\sigma(1)}F^\rho_{\langle u,\bv_\rho\rangle}=\{e \in \G\mid u \in P_{e, \sigma}\},$$ 
and 
$$H^0(X_\Sigma, \E)_u =\bigcap_{\rho\in \Sigma(1)}F^\rho_{\langle u,\bv_\rho\rangle}= \{e \in \G\mid u \in P_{e}\}.   $$ 

    \end{definition}

 We denote the rank of the flat $H^0(U_\sigma, \E\!\!\mid_{U_\sigma})_u$ and $H^0(X_\Sigma, \E)_u$ by $h^0(U_\sigma, \E\!\!\mid_{U_\sigma})_u$ and $h^0(X_\Sigma, \E)_u$ respectively. 
\begin{remark} \label{h^0}
    Alternatively, following from Definition \ref{def-matroid-vb-1},
    $$H^0(U_\sigma, \E\!\!\mid_{U_\sigma})_u = \Span\{i\in B_\sigma \mid \langle u,\bv_\rho \rangle \leq \langle u_{\sigma,i},\bv_\rho\rangle, \forall \rho\in \sigma(1)\},$$
    and
$$h^0(U_\sigma, \E\!\!\mid_{U_\sigma})_u = \vert  \{i \in B_\sigma\mid \langle u,x\rangle \leq \langle u_{\sigma,i},x\rangle, \forall x\in\sigma\}\vert $$
\end{remark}

\begin{definition}[Equivariant Euler characteristic of a tropical vector bundle]   \label{def-equiv-Euler-matriod-vb}
For a character $u$ we define the Euler characteristic $\chi(X_\Sigma, \E)_u$ by:
\begin{align*}
\chi(X_\Sigma,\E)_u &= \sum_{\sigma \in \Sigma} (-1)^{\codim(\sigma)} \rank H^0(U_\sigma, \E\!\!\mid_{U_\sigma})_u,\\
&= \sum_{\sigma \in \Sigma} (-1)^{\codim(\sigma)} \rank\;\{e \in \G \mid u \in P_{e, \sigma} \}.    
\end{align*}
{Moreover, we let:
$$\chi(X_\Sigma,\E) = \sum_{u \in M} \chi(X_\Sigma,\E)_u.$$}
\end{definition}

\section{Khovanskii-Pukhlikov Hirzebruch-Riemann-Roch theorem for polytopes}  \label{sec-KhPukh}
In this section, we give a brief overview of Khovanskii-Pukhlikov theory of convex chains and their combinatorial Hirzebruch-Riemann-Roch formula for polytopes (\cite{Khovanskii-Pukhlikov-1,Khovanskii-Pukhlikov-2}).

Let $\mathcal{P}$ denote the set of all convex polytopes in the real vector space $M_\mathbb{R} = M \otimes \R$. For $P, Q \in \mathcal{P}$, we write the Minkowski sum of $P$ and $Q$ as $P \oplus Q  = \{x+y \mid x\in P, y\in Q\}$.

 \begin{definition}
        A \emph{convex chain} is a function, $\alpha : M_\mathbb{R} \to \mathbb{Z}$ of the form $\alpha = \displaystyle\sum_{i=1}^k n_i\mathbbm{1} _{P_i}$ where $n_i\in \mathbb{Z}$, $ P_i\in \mathcal{P}$ and $\mathbbm{1}_{P_i}$ denotes the indicator function of the polytope $P_i$. If each $P_i$ is a convex lattice polytope, we call the convex chain a \emph{lattice} convex chain. 
   
 \end{definition}
           We denote the additive group of convex chains (respectively lattice convex chains) by $Z(M_\R)$ (respectively by $Z(M)$). 
 
\begin{definition}
          For a convex chain, $\alpha = \displaystyle\sum_{i=1}^k n_i \mathbbm{1}_{P_i}$, the number $\displaystyle\sum_{i=1}^k n_i$ depends only on $\alpha$ and is called the \emph{degree} of $\alpha$.
 
\end{definition}

\begin{proposition} \label{prop-convoltuion}
        Minkowski summation of polytopes extends in a unique way
 to a bilinear operation on the group of convex chains.
 $$*: Z(M_\R)\times Z(M_\R) \to Z(M_\R), \,\,\mathbbm{1}_A *\mathbbm{1}_B= \mathbbm{1}_{A\,\oplus\, B}.$$

\end{proposition}

 The set $Z(M_\R)$ with the operations of $+$ and $*$ is called the \emph{algebra of convex chains}. The multiplicative identity is the function $\mathbbm{1}_{\{0\}}$. An invertible element of degree $1$ is called a \emph{virtual polytope}. We denote the multiplicative group of virtual polytopes by $\mathcal{P}^*$. In fact, $\mathcal{P}^*$ is an infinite dimensional real vector space.

\begin{proposition}
    For a polytope $P \subset M_\mathbb{R}$,
    $$(-1)^{\dim(P)} \mathbbm{1}_{\Int(P)} = \displaystyle\sum_{\Delta \in \Gamma(P)} (-1)^{\dim(\Delta)} \mathbbm{1}_\Delta $$ where  $\Gamma(P)$
    is the set of all faces of $P$.
\end{proposition}

\begin{theorem}[Minkowski inversion]
       $$(\mathbbm{1}_P)^{-1} = (-1)^{\dim(P)}\mathbbm{1}_{\Int(-P)}= \displaystyle\sum_{\Delta \in \Gamma(-P)} (-1)^{\dim(\Delta)}\mathbbm{1}_\Delta.$$
\end{theorem}

We recall that a support function of a polytope $P$ is a piecewise linear function $\h_P :N_\mathbb{R} \to \mathbb{R}$,
$$\h_P(x) = \displaystyle \max_{u\in P} \langle u,x\rangle.$$

\begin{definition}
        The \emph{support function} of a convex chain $\alpha =\displaystyle\sum_{i=1}^k n_i \mathbbm{1}_{P_i}$ is the function 
    $$\Phi_\alpha: N_\mathbb{R} \to \mathbb{Z}[\mathbb{R}],  \,\,x\mapsto \displaystyle\sum_{i=1}^k n_i[\h_{P_i}(x)].$$ Here $\Z[\R]$ denotes the group ring of the additive group of $\R$, consisting of formal linear combinations of real numbers with integer coefficients.
    \end{definition}

\begin{theorem}  \label{th-convex-chain-supp-function}
    The association of a convex chain to its support function is an algebra isomorphism of $Z(M_\R) $ and the algebra of all multi-valued piecewise linear functions. Moreover, under this isomorphism, the subspace of virtual polytopes corresponds to the subspace of (one-valued) support functions.   
\end{theorem}

For later use, we record the following classic theorem which represents the indicator function of a convex polytope $P$ as an alternating sum of indicator functions of cones at faces of $P$. For a proof see \cite{Beck-Haase-Sottile}.
\begin{theorem}[Brianchon-Gram]\label{BG} For polytope $P$,

$$\mathbbm{1}_P = \sum_{F\preceq P} (-1)^{\dim(F)}\mathbbm{1}_{C(F)}$$
where $C(F)$ is the tangent cone of the face $F$ of $P$, that is, $C(F)$ is the intersection of all the supporting half-spaces of $P$ having the face $F$ on their boundaries.
\end{theorem}

\begin{remark} \label{rem-Brianchon-Gram}
    If $P$ is a polytope with (outer) normal fan $\Sigma$, then there are constants $d_\rho$ for $\rho\in \Sigma(1)$ for which $P = \{u \in M_\R\,|\, \langle u , \bv_\rho\rangle \leq d_\rho,\forall \rho\in \Sigma(1)\}$.
    For a cone $\sigma \in \Sigma$, the tangent cone of the face $F$ of $P$ corresponding to $\sigma$, is given by: $C_\sigma = \{u \in M_R \mid \langle u , \bv_\rho\rangle \leq d_\rho,\forall \rho\in \sigma(1)\}$. Then, the Brianchon-Gram formula can be written as:
    $$\mathbbm{1}_P = \sum_{\sigma\in\Sigma}(-1)^{\codim (\sigma)} \mathbbm{1}_{C_\sigma}.$$

    In fact, the Brianchon-Gram theorem extends to virtual polytopes as well (\cite[\S 4, Proposition 2]{Khovanskii-Pukhlikov-1}). More precisely, let us fix a projective fan $\Sigma$. Let $\mathcal{P}^*(\Sigma)$ be the subgroup of $(\mathcal{P}^*, *)$ generated by all the indicator functions of the convex polytopes normal to $\Sigma$, that is, convex polytopes whose normal fan is $\Sigma$. In fact, $\mathcal{P}^*(\Sigma)$ is a finite dimensional real vector space. For $\alpha \in \mathcal{P}^*(\Sigma)$ with support function $\h$, we have:
    $$\alpha = \sum_{\sigma \in \Sigma} (-1)^{\codim(\sigma)} \mathbbm{1}_{C_{\sigma}}, $$
    where $C_\sigma$ is the cone defined by $C_{\sigma} = \{u\in M_\R \mid \langle u,\bv_\rho\rangle \leq \h(\bv_\rho), \,\forall \rho\in \sigma(1)\rangle\}$. 
\end{remark}

The vector space $\mathcal{P}^*(\Sigma)$ has natural coordinates corresponding to rays in $\Sigma$. Let $\bv_1,\ldots, \bv_s$ be the primitive vectors of rays of $\Sigma$. Let $\alpha$ be a virtual polytope with support function $\h_\alpha$. The piecewise linear function is uniquely determined by its values $z_i = \h_\alpha(\bv_i)$ for $ 1\leq i\leq s$. We refer to the $z_i$ as \emph{support numbers} of $\alpha$ and think of them as coordinates on the vector space $\mathcal{P}^*(\Sigma)$. For any $z=(z_1, \ldots, z_s)$ we denote the virtual polytope in $\mathcal{P}^*(\Sigma)$ with support numbers $z_i$ by $\alpha[z]$. 


\begin{definition} For a continuous function, $f:M_\R \to \R$,
    The \emph{integral over a convex chain} $\alpha$ is the number
    $$I_f(\alpha) = \int_{M_\R} \alpha(u)f(u) \,du.$$
    The \emph{lattice sum over a convex chain} $\alpha$ is the number
    $$S_f(\alpha) = \sum_{u\in M}\alpha(u)f(u).$$

If $f=1$, we denote $S_f(\alpha)$ and $I_f(\alpha)$ simply by $S(\alpha)$ and $I(\alpha)$ respectively.
\end{definition}

For a convex chain $\alpha$, a virtual polytope $\beta$ and the coordinate $z= (z_1,\ldots,z_s)$, we have a pair of maps 
$z \mapsto I_f(\alpha * \beta[z])$ and $z \mapsto S_f(\alpha*\beta[z])$. The following remarkable theorem of Khovanskii and Pukhlikov gives a connection between $S_f(\alpha)$ and $I_f(\alpha)$ (\cite{Khovanskii-Pukhlikov-2}). 

\begin{theorem}[Khovanskii-Pukhlikov]\label{Thm-Kh-Pu}
    In the above notation, for $c=(c_1,\ldots,c_s) \in \mathbb{Z}^s$, and a polynomial $f$,
    $$\Todd\left(\frac{\partial}{\partial z}\right) I_f(\alpha* \beta[z])\big|_{z=c} = S_f(\alpha *\beta[c]).$$
\end{theorem} 


The motivation for the above theorem was the Hirzebruch-Riemann-Roch theorem for line bundles on a projective toric variety.

\begin{remark}
Khovanskii-Pukhlikov's results is actually more general and applies to quasi-polynomial, that is, a combination of polynomials and exponential functions. The Khovanskii-Pukhlikov Riemann-Roch is sometimes called a multi-dimensional exact Euler-Maclaurin formula.  
\end{remark}

\section{Euler characteristic of a tropical vector bundle}
In this section, we introduce two (equivalent) constructions for a convex chain associated to a tropical vector bundle $\E$ with piecewise linear map $\Phi_\E : |\Sigma| \to \widetilde{\Berg}(\M)$ which encodes the equivariant Euler characteristic of the tropical bundle. 

\subsection{Convex chain associated to a tropical vector bundle}  \label{subsec-cc-trop-vb} 
Throughout this section, let $\E$ be a rank $r$ tropical vector bundle with decreasing filtrations $(F_i^\rho)_{i\in\Z}$ for all $\rho \in \Sigma(1)$
 satisfying the compatibility condition. Let $u(\sigma) = \{ u_{\sigma,1},\ldots, u_{\sigma,r}\}$ be the multi-set of characters and $B_\sigma$ be a basis in $\mathcal{M}$ corresponding to $\sigma \in \Sigma$.
 
 \begin{definition}[Convex chain of a tropical vector bundle]\label{multi-valued support}
The \emph{(multi-valued) support function} of the tropical vector bundle, $\E$ is the map $\h_\E :|\Sigma| \to \Z[\R]$,
 $$\h_\E(x) = \sum_{i=1}^r [\langle u_{\sigma,i},x\rangle] \quad\text{for}\,\,x\in\sigma.$$
 We denote the convex chain corresponding to the support function of a tropical vector bundle, $\E$ by $\alpha_\E$.
 \end{definition}

\subsection{Construction of convex chain using a split resolution} \label{subsec-split-res}
In \cite{Klyachko}, for a toric vector bundle $\E$ on a smooth complete toric variety $X_\Sigma$, Klyachko constructs a resolution $0 \to \E \to \F_0 \to \cdots \to \F_n \to 0$ by split toric vector bundles on $X_\Sigma$, where $n = \dim(X_\Sigma)$. Recall that a toric vector bundle is split if it is equivariantly isomorphic to a direct sum of toric line bundles. 

In this section we extend the construction of this resolution to tropical vector bundles. Note that we do not yet have a notion of morphism of tropical vector bundles, so strictly speaking we cannot talk about a resolution of tropical vector bundles. Rather we construct a sequence of split tropical vector bundles $\F_k$, $k=0, \ldots, n$, such that the alternating sum of the equivariant $K$-classes of the $\F_k$ equals the equivariant $K$-class of $\E$ (Theorem \ref{th-K-class-resolution}). This then gives an alternative construction of the convex chain $\alpha_\E$ associated to a tropical vector bundle $\E$ (encoding the equivariant Chern roots of $\E$).


We begin by recalling the notion of equivariant $K$-class of a tropical vector bundle. Let $\E$ be a tropical vector bundle on $X_\Sigma$ given by a
piecewise linear map
$$\Phi_\E : |\Sigma| \to \widetilde{\Berg}(\M).$$
Following \cite{KM-trop-vb}, the equivariant $K$-class of $\E$ is the piecewise exponential function
$$K_{\E} : |\Sigma| \to \R$$
defined by
$$K_{\E} = \exp_{\M} \circ \Phi_\E,$$
where $\exp_{\M}$ denotes the sum of exponential functions on $\widetilde{\Berg}(\M)$, which on each apartment $\widetilde{A_\sigma} \simeq \R^r$ is given by
$$\exp(x_1,\dots,x_r) = \sum_{i=1}^r e^{x_i}.$$
Equivalently, for each maximal cone $\sigma \in \Sigma$, there exist characters $u_{\sigma,1},\dots,u_{\sigma,r}$ such that
$$\Phi_\E|_\sigma(x)=(\langle u_{\sigma,1},x\rangle,\dots,\langle u_{\sigma,r},x\rangle),$$
and hence for $x\in\sigma$
$$K_{\E}(x)=\sum_{i=1}^r \exp\big(\langle u_{\sigma,i},x\rangle\big).$$

The construction of the resolution depends on the choice of a piecewise linear function $f$. Let $f: |\Sigma| \to \R$ be a piecewise linear function such that $F^\rho_i = 0$ if $i> f(\bv_\rho)$. For each cone $\tau \in \Sigma$ let $f$ be given by the linear function $\langle u_{f,\tau}, \cdot \rangle$. For $0 \leq k \leq n$,
we will construct a tropical vector bundle $\F_k$. The filtrations associated to $\F_k$ are given by:
$$G^\rho_i =\bigoplus_{\codim(\sigma)=k}G_{\sigma,i}^\rho $$
where 
$$G_{\sigma,i}^\rho=\begin{cases}
    F_i^\rho &\quad\text{if}\,\,\rho \in \sigma(1)\\
    \G &\quad\text{if}\,\,\rho\not\in \sigma(1), f(\bv_\rho)\geq i\\
    0&\quad\text{if}\,\,\rho\not\in \sigma(1) , f(\bv_\rho)<i.
\end{cases}$$

In order for the $G^\rho_i$ to define a tropical vector bundle, we need to show that they satisfy the compatibility condition in Definition \ref{def-matroid-vb-1}. For $\sigma, \tau \in \Sigma$, let $u^\sigma(\tau)= \{u^\sigma_{\tau,1},\ldots,u
^\sigma_{\tau,r}\}$ be the multi-set of characters defined by the property:
$$\langle u^\sigma_{\tau,j},\bv_\rho\rangle = \begin{cases}
    \langle u_{f,\tau},\bv_\rho\rangle=f(\bv_\rho)&\quad\text{if}\,\, \rho \in \tau \setminus \sigma\\
    \langle u_{\sigma,j},\bv_\rho \rangle & \quad\text{if}\,\, \rho\in \tau \cap \sigma.
\end{cases}$$
The $u^\sigma_{\tau,j}$ exist because of the smoothness assumption of $\Sigma$.
Moreover, the equivariant $K$-class of $\F_k$ is then
$$K_{\F_k}(x) = \sum_{\codim(\sigma)=k}\big(\sum_{j=1}^r \exp(\langle u^\sigma_{\tau,j},x\rangle)\big).$$


\begin{proposition} \label{prop-F-k-compatible}
    In the above notation, the filtrations $G^\rho_{i}$ satisfy the compatibility condition with respect to the multi-sets of characters $u^\sigma(\tau)$ and define a split tropical vector bundle.
\end{proposition}
\begin{proof}
Let $\sigma\in\Sigma$ with $\codim(\sigma)=k$ and $\tau\in \Sigma$ be a maximal cone.
Since $\E$ is a tropical vector bundle,
there exist a basis $B_\tau=\{b_{\tau,1},\dots,b_{\tau,r}\}$ and characters
$u(\tau)=\{u_{\tau,1},\dots,u_{\tau,r}\}\subset M$ such that for every $\rho\in\tau(1)$,
$$F_i^\rho=\Span\{ b_{\tau,j} \mid \langle u_{\tau,j}, \bv_\rho\rangle \geq i \}.$$
For $\rho\in\sigma(1)\cap\tau(1)$, the
$\sigma$-component of the filtration $(G^{\rho}_{\sigma,j})_j$ satisfies the
compatibility condition because the filtrations of $\E$ satisfy the condition.
For $\rho\in\tau(1)\setminus\sigma(1)$, by definition
$$\langle u^{\sigma}_{\tau,j}, \bv_\rho\rangle
=\langle u_{f,\tau}, \bv_\rho\rangle= f(\bv_\rho),$$
and hence
$$G^{\rho}_{\sigma,i}=
\begin{cases}
\G, & f(\bv_\rho)\geq i,\\
0, & f(\bv_\rho)< i.
\end{cases}$$
A direct check shows that for every $\rho\in\tau(1)$,
$$G^{\rho}_{\sigma,i} = \Span\{ b_{\tau,j} \mid \langle u^{\sigma}_{\tau,j}, \bv_\rho\rangle \geq i \}.$$
Taking the direct sum over all $\sigma$ with $\codim(\sigma)=k$, the flats
$G_i^\rho$ are spanned by subsets of the direct-sum basis. Thus, the
compatibility condition holds and the filtrations define a tropical vector bundle.

We note that by the compatibility condition, for a given $\sigma$, all the flats $G^\rho_{\sigma, i}$ are adapted to the same basis. This implies that the all the $G^\rho_i$ are adapted to the same bases and hence the tropical vector bundle $\F_k$ is split. 
\end{proof}
\begin{theorem}  \label{th-K-class-resolution}
Let $\E$ be a rank $r$ tropical vector bundle on a smooth projective $X_\Sigma$ with multiset of
characters $u(\tau)=\{u_{\tau,1},\dots,u_{\tau,r}\}$ on each cone $\tau\in\Sigma$, and let
$f:|\Sigma|\to \mathbb{R}$ be as in \S4.2. For $\sigma,\tau\in\Sigma$, let
$u^\sigma(\tau)=\{u^\sigma_{\tau,1},\dots,u^\sigma_{\tau,r}\}$ be the multiset of characters defined by
$$\langle u^\sigma_{\tau,j}, \bv_\rho\rangle=
\begin{cases}
f(\bv_\rho) & \text{if }\rho\in \tau(1)\setminus\sigma(1),\\
\langle u_{\sigma,j}, \bv_\rho\rangle & \text{if }\rho\in \tau(1)\cap\sigma(1).
\end{cases}$$
For the tropical vector bundles $\F_k$ constructed in \S4.2,
$$K_\E(x)=\sum_{\sigma\in\Sigma}(-1)^{\codim(\sigma)}\sum_{j=1}^r
\exp(\langle u^\sigma_{\tau,j},x\rangle)
=\sum_{k=0}^n(-1)^k K_{\F_k}(x).$$
\end{theorem}
\begin{proof}
Fix a cone $\tau\in\Sigma$. For a ray $\rho\in \tau(1)$, let $\Sigma(\rho)$ be the set of cones of $\Sigma$ that contain $\rho$.
Evaluating the alternating sum at the primitive generator $\bv_\rho$, we have
$$\sum_{k=0}^n (-1)^k K_{\F_k}(\bv_\rho)
=\sum_{\sigma\in\Sigma}(-1)^{\codim(\sigma)}\sum_{j=1}^r
\exp(\langle u^\sigma_{\tau,j}, \bv_\rho\rangle).
$$
Split the sum into $\sigma\in\Sigma(\rho)$ and $\sigma\notin\Sigma(\rho)$:
$$=\sum_{\sigma\in\Sigma(\rho)}(-1)^{\codim(\sigma)}\sum_{j=1}^r
\exp(\langle u^\sigma_{\tau,j}, \bv_\rho\rangle)
+\sum_{\sigma\notin\Sigma(\rho)}(-1)^{\codim(\sigma)}\sum_{j=1}^r
\exp(\langle u^\sigma_{\tau,j}, \bv_\rho\rangle).
$$
If $\sigma\in\Sigma(\rho)$ then $\rho\in\sigma(1)\cap\tau(1)$, hence
$\langle u^\sigma_{\tau,j}, \bv_\rho\rangle
=\langle u_{\sigma,j}, \bv_\rho\rangle$.
If $\sigma\notin\Sigma(\rho)$ then $\rho\in\tau(1)\setminus\sigma(1)$, hence
$\langle u^\sigma_{\tau,j}, \bv_\rho\rangle=f(\bv_\rho)$ for all $j$.
Therefore
$$\sum_{k=0}^n (-1)^k K_{\F_k}(\bv_\rho)=
\Big(\sum_{\sigma\in\Sigma(\rho)}(-1)^{\codim(\sigma)}\Big)\,K_\E(\bv_\rho) +
\Big(\sum_{\sigma\notin\Sigma(\rho)}(-1)^{\codim(\sigma)}\Big)\,
r\,\exp(f(\bv_\rho)).$$
Using the inclusion--exclusion identities for the posets $\Sigma(\rho)$ and $\Sigma$,
$$\sum_{\sigma\in\Sigma(\rho)}(-1)^{\codim(\sigma)}=1,
\qquad
\sum_{\sigma\in\Sigma}(-1)^{\codim(\sigma)}=1,$$
hence
$$\sum_{k=0}^n (-1)^k K_{\F_k}(\bv_\rho)=K_\E(\bv_\rho)
\qquad\text{for every }\rho\in\tau(1).$$

Since $\Sigma$ is smooth, $\sum_j \exp(\langle u_j,x\rangle)$ is determined by its values on the rays of $\tau$, and the identity holds for all $x\in\tau$.
\end{proof}

Motivated by Theorem \ref{th-K-class-resolution}, we give an alternative way to associate a multi-valued support function and hence a convex chain to a tropical vector bundle $\E$ as follows. We know that each $\F_k$ is split and hence can be written as a sum of toric line bundles $\LL_{i,k}$. Let $\h_{i,k}$ be the (one-valued) support function of $\LL_{i, k}$. Define the multi-valued support function $\h_\E: |\Sigma| \to \Z[\R]$ by:
\begin{equation}  \label{equ-alt-h-E}
\h_\E(x) = \sum_{k=0}^r \sum_i (-1)^k [h_{k, i}(x)].    
\end{equation}

We then let $\alpha_\E:M_\R \to \R$ to be the convex chain associated to the multi-valued support function $\h_\E$ (see Theorem \ref{th-convex-chain-supp-function}).

\begin{corollary}
The convex chain $\alpha_\E$ coincides with the convex chain defined in Definition \ref{multi-valued support}.    
\end{corollary}
\begin{proof}
It suffices to show that $\h_\E$ given in the equation \eqref{equ-alt-h-E} coincides with the one in Definition \ref{multi-valued support}. The proof of this is identical to the proof of Theorem \ref{th-K-class-resolution}.   
\end{proof}

\begin{example}
Let $\Sigma$ be the fan of $\P^2$ with rays, $\rho_1,\rho_2, \rho_3$ generated by
$\bv_1=(1,0), \bv_2=(0,1), \bv_3=(-1,-1),$
and maximal cones
$$\sigma_{12}=\cone(\bv_1,\bv_2),\qquad
\sigma_{23}=\cone(\bv_2,\bv_3),\qquad
\sigma_{13}=\cone(\bv_1,\bv_3).$$
Let $\M=U_{2,3}$ be the uniform matroid of rank $2$ on $\G=\{1,2,3\}$ and let $\E$ be the tropical vector bundle defined by the diagram
$$\begin{array}{c|ccc}
 & 1 & 2 & 3 \\ \hline
\rho_1 & 1 & 0 & 0 \\
\rho_2 & 0 & 1 & 0 \\
\rho_3 & 0 & 0 & 1
\end{array}$$
with compatible filtrations
$$F_i^{\rho_k}=
\begin{cases}
\G & i< 1,\\
\{k\} & i=1,\\
0 &  1<i.
\end{cases}$$
For each maximal cone $\sigma$, we have a basis
$B_{\sigma_{12}}=\{1,2\}, B_{\sigma_{23}}=\{2,3\}, B_{\sigma_{13}}=\{1,3\},$
and characters
$$u_\E(\sigma_{12})=\{(1,0),(0,1)\},\quad
u_\E(\sigma_{23})=\{(-1,1),(-1,0)\},\quad
u_\E(\sigma_{13})=\{(1,-1),(0,-1)\}.$$
For maximal cones $\tau,$ the characters of $\sigma$-summand of $\F_0$ on the chart $U_\tau$  are:
\begin{center}
\renewcommand{\arraystretch}{1.3}
\begin{tabular}{c|c|c|c}
$\tau \backslash \sigma$
& $\sigma_{12}$ & $\sigma_{23}$ & $\sigma_{13}$ \\
\hline
$\sigma_{12}$
& $\{(1,0),(0,1)\}$
& $\{(0,1),0\}$
& $\{0,(1,0)\}$ \\

$\sigma_{23}$
& $\{0,(-1,1)\}$
& $\{(-1,1),(-1,0)\}$
& $\{(-1,0),0\}$ \\

$\sigma_{13}$
& $\{(1,-1),0\}$
& $\{0,(0,-1)\}$
& $\{(0,-1),(1,-1)\}$
\end{tabular}
\end{center}
Summing over maximal cones, $\F_0$ is a rank $6$ tropical vector bundle with multi-set of characters
$$u_{F_0}(\tau)=\bigsqcup_{\sigma\ \text{maximal}} u^\sigma(\tau),$$
and
$$\begin{aligned}
u_{\F_0}(\sigma_{12})&=\{(1,0),(1,0),(0,1),(0,1),0,0\},\\
u_{\F_0}(\sigma_{23})&=\{(-1,1),(-1,1),(-1,0),(-1,0),0,0\},\\
u_{\F_0}(\sigma_{13})&=\{(1,-1),(1,-1),(0,-1),(0,-1),0,0\}.
\end{aligned}$$
Similarly, for $\F_1$ and rays $\sigma=\rho_i$ characters $u^\sigma_\tau$ on $U_\tau$ for maximal cones $\tau$ are:
\begin{center}
\renewcommand{\arraystretch}{1.3}
\begin{tabular}{c|c|c|c}
$\tau \backslash \sigma$
& $\rho_1$ & $\rho_2$ & $\rho_3$ \\
\hline
$\sigma_{12}$
& $\{(1,0),0\}$
& $\{(0,1),0\}$
& $\{0,0\}$ \\

$\sigma_{23}$
& $\{0,0\}$
& $\{(-1,1),0\}$
& $\{(-1,0),0\}$ \\

$\sigma_{13}$
& $\{(1,-1),0\}$
& $\{0,0\}$
& $\{(0,-1),0\}$
\end{tabular}
\end{center}

The multi-set of characters of $\F_1$ on maximal cones are
$$\begin{aligned}
u_{\F_1}(\sigma_{12})&=\{(1,0),(0,1),0,0,0,0\},\\
u_{\F_1}(\sigma_{23})&=\{(-1,1),(-1,0),0,0,0,0\},\\
u_{\F_1}(\sigma_{13})&=\{(1,-1),(0,-1),0,0,0,0\}.
\end{aligned}$$

And for all cones $\tau$
$$u_{\F_2}(\tau)=\{0,0\}.$$

Let $\h_\E$ and $\h_{\F_k}$ be the support function as in Definition \ref{multi-valued support}. In particular, on $\sigma_{12}$,
$$\begin{aligned}
\h_{\F_0}(x)&=2[\langle(1,0),x\rangle]+2[\langle(0,1),x\rangle]+2[0],\\
\h_{\F_1}(x)&=[\langle(1,0),x\rangle]+[\langle(0,1),x\rangle]+4[0],\\
\h_{\F_2}(x)&=2[0],
\end{aligned}$$
so
$$\h_{\F_0}(x)-\h_{\F_1}(x)+\h_{\F_2}(x)=[\langle(1,0),x\rangle]+[\langle(0,1),x\rangle]=\h_\E(x).$$
The same computation holds on $\sigma_{23}$ and $\sigma_{13}$, hence
$$\h_{\F_0}-\h_{\F_1}+\h_{\F_2}=\h_\E,$$
and the two constructions coincide.

\end{example}

\subsection{Euler characteristic of a tropical vector bundle}
It is well known that for a virtual polytope with convex chain $\alpha$ and the associated toric line bundle $\mathcal{L}_\alpha$, $$\chi(X_\Sigma,\mathcal{L}_\alpha) = S(\alpha)=\displaystyle\sum_{u\in M} \alpha(u),$$
(see for example \cite[Proposition 3.3]{CHK}). Similar to toric line bundles, we introduce a combinatorial description of the Euler characteristic of a tropical vector bundle by showing that the invariance of Euler characteristic of tropical vector bundle under a toric pull-back.
\begin{remark}
     The notion of \emph{pull-back} of a tropical vector bundle makes sense. For a refinement $\Sigma'$ of $\Sigma$, the pull-back bundle is defined by a diagram with more rows. The entries in the diagram corresponding to the new rays are given by the piecewise linearity of the map $\Phi_\E$.
     
       Explicitly, let $\E$ be a tropical vector bundle on $X_\Sigma$ with a matroid $\M$. Let $\varphi: X_{\Sigma'}\to X_\Sigma$ be the toric morphism that is determined by a morphism between the cocharacter lattices $\overline{\varphi}:N'\to N$. Then, we define the pull-back tropical vector bundle $\varphi^*\E$ over $X_{\Sigma'}$ by the data, $B_{\sigma'} = B_\sigma$ and $u(\sigma')=u(\sigma)$ where $\overline{\varphi}(\sigma')\subset\sigma$.
\end{remark}

\begin{lemma}\label{h^0sigma}
     Let $\sigma$ be a cone in $N_\mathbb{R}$ and $\sigma^\vee+u = \{p \in M_\mathbb{R} \mid \langle u,x \rangle \leq \langle p,x\rangle, \,\forall x\in \sigma\}$ be the dual cone $\sigma^\vee$ translated by $u$. Then,
$$h_\sigma^0(u):=h^0(U_\sigma, \E)_u  = \displaystyle\sum_{i=1}^r \mathbbm{1}_{\sigma^\vee+u} (u_{\sigma,i}).$$
\end{lemma}
\begin{proof}
     It suffices to prove for $u=0$ by translation.
        By Remark \ref{h^0} we have 
        $$h_\sigma^0(u) = \vert\{ b_i \in B_\sigma\mid \langle u,x\rangle \leq \langle u_{\sigma,i},x\rangle \,\,, \forall x\in\sigma\}\vert.$$
        Then, clearly
$$h_\sigma ^0(0)  = \sum_{i=1}^r \mathbbm{1}_{\sigma^\vee} (u_{\sigma,i}).$$
\end{proof}

    

\begin{lemma}\label{lemma-BG-polyhedron}
Let $P \subset \R^n$ be a  polyhedron with no linearity space, that is, it does not contain any line. Then
$$
\sum_{F \preceq P} (-1)^{\dim(F)}
=\begin{cases}
1 & \text{$P$ is bounded},\\
0 & \text{$P$ is unbounded}.
\end{cases}
$$
\end{lemma}
\begin{proof}
If $P$ is bounded, the left hand side computes the Euler characteristic of $P$ and hence is equal to $1$ since $P$ is contractible. Now we assume $P$ is unbounded. For $x \in N_\R$ and $R > 0$ let:
$$H_{x, R} = \{u \in M_\R \mid \langle u,x\rangle=R\} \quad \text{and}\quad H^+_{x, R} = \{u \in M_\R \mid \langle u,x\rangle\geq R\}.$$
 Take $x\in N_\R$ and $R>0$ such that $P' := P \cap H^+_{x, R}$ is the polytope obtained by intersecting $P$ and  $H^+_{x,R}$.
New faces of $P'$ correspond to precisely unbounded faces of $P$. Hence,
$$1=  \sum_{F\preceq P'}(-1)^{\dim(F)} = \sum_{F\preceq P}(-1)^{\dim(F)}+\sum_{\substack{F\preceq P\\ \text{unbounded}}}(-1)^{\dim(F)-1}=  \sum_{F\preceq P}(-1)^{\dim(F)}+1$$
because $\displaystyle\sum_{\substack{F\preceq P \\ \text{unbounded}}}(-1)^{\dim(F)-1}$ computes the Euler characteristic of the polytope $P\cap H_{x,R}.$
\end{proof}

\begin{proposition}\label{Prop-BG-Cone supp}
 Let $\Sigma$ be a fan whose support is a convex cone $C:=|\Sigma|$. Then,
$$\sum_{\sigma\in \Sigma} (-1)^{\codim(\sigma)}\mathbbm{1}_{\sigma^\vee} = (-1)^{n}\mathbbm{1}_{-\relint(C^\vee)}.$$
\end{proposition}
\begin{proof}
Fix $u\in M_\R$ and set $P(u):=C\cap H^+(u)$, where
$H^+(u):=\{x\in N_\R:\langle u,x\rangle \geq 1\}$.
Then $P(u)$ is bounded if and only if $u\in-\relint(C^\vee)$.
By Lemma \ref{lemma-BG-polyhedron},
$$
\sum_{F\preceq P(u)}(-1)^{\dim F}=\mathbbm{1}_{-\relint(C^\vee)}(u).
$$
Note that $u\in \sigma^\vee$ if and only if $\sigma$ corresponds to a face of $P(u)$. Hence,
$$\sum_{\sigma\in\Sigma}(-1)^{\codim(\sigma)}\,\mathbbm{1}_{\sigma^\vee}(u)
=
(-1)^n \sum_{F\preceq P(u)}(-1)^{\dim F}
=
(-1)^n\,\mathbbm{1}_{-\relint(C^\vee)}(u).
$$    
\end{proof}


\begin{proposition}\label{lemma-BG-cone}
    Let $\Sigma$ be a fan and $\Sigma'$ be a refinement of $\Sigma$. Let
    $\phi:\Sigma'\to \Sigma$
    denote the map sending a cone $\tau\in\Sigma'$ to the smallest cone $\sigma\in \Sigma$ that contains $\tau$. Then, for a cone $\sigma\in \Sigma$ we have    
\begin{equation}\label{eq-BG-fiber}
(-1)^{\codim(\sigma)}\mathbbm{1}_{\sigma^\vee} = \sum_{\tau \in \phi^{-1}(\sigma)}(-1)^{\codim(\tau)}\mathbbm{1}_{\tau^\vee}.    
\end{equation}
\end{proposition}
\begin{proof}
For a fixed $\sigma\in\Sigma,$ let $\Sigma_\sigma$ denote the face fan of $\sigma$, and let
$\Sigma'_\sigma := \{\tau\in\Sigma' \mid \tau\subseteq\sigma\}$
be the restriction of $\Sigma'$ to the cones contained in $\sigma$. Then $\Sigma'_\sigma$ is a subdivision of $\Sigma_\sigma$. By Proposition \ref{Prop-BG-Cone supp} applied to $\Sigma_\sigma$ and $\Sigma'_\sigma,$
$$\sum_{F\preceq\sigma} (-1)^{\codim(F)}\mathbbm{1}_{F^\vee}= (-1)^n\mathbbm{1}_{-\relint(\sigma^\vee)} =\sum_{\tau\in \Sigma'_\sigma}(-1)^{\codim(\tau)}\mathbbm{1}_{\tau^\vee}.$$
Since each $\tau \in \Sigma'_\sigma$ belongs to $\phi^{-1}(F)$ for a unique face $F\preceq \sigma$,
\begin{equation}\label{eq-BG-local}
\sum_{F\preceq\sigma}(-1)^{\codim(F)}1_{F^\vee}=\sum_{F\preceq\sigma}
\sum_{\tau\in\phi^{-1}(F)}(-1)^{\codim(\tau)}
1_{\tau^\vee}
.
\end{equation}
By M\"obius inversion on $\Sigma_\sigma$, 
$$(-1)^{\codim(\sigma)}\mathbbm{1}_{\sigma^\vee} = \sum_{\tau \in \phi^{-1}(\sigma)}(-1)^{\codim(\tau)}\mathbbm{1}_{\tau^\vee}.$$
\end{proof}
\begin{theorem}  \label{th-toric-pull-back}
    The equivariant Euler characteristic of a tropical vector bundle over a complete toric variety is invariant under a toric pull-back.
\end{theorem}
\begin{proof}
    Let $\Sigma$ be a complete fan with a refinement $\phi:\Sigma'\to \Sigma$ where $\phi(\sigma')$ is the smallest cone in $\Sigma $ that contains $\sigma'$. For $\tau \subset \sigma$, let $u_{\sigma,i} = u_{\tau,i}$ be the characters of the bundle $\E$ over $X_\Sigma$ and its pull-back bundle $\E'$ over $X_{\Sigma'}$. Evaluating Equation \eqref{eq-BG-fiber} at the characters and summing over $1\leq i\leq r$, we have
$$    (-1)^{\codim(\sigma)}\sum_{i=1}^r\mathbbm{1}_{\sigma^\vee}(u_{\sigma,i}) = \sum_{\tau \in \phi^{-1}(\sigma)}(-1)^{\codim(\tau)}\sum_{i=1}^r\mathbbm{1}_{\tau^\vee}   (u_{\tau,i}).
$$
By Lemma \ref{h^0sigma} and by translation,
$$(-1)^{\codim(\sigma)}h^0_\sigma(u)=\sum_{\tau \in \phi^{-1}(\sigma)}(-1)^{\codim(\tau)}h^0_\tau(u).$$
Taking sum over all $\sigma \in \Sigma$, we have
$$\chi(X_\Sigma,\E)_u = \sum_{\sigma \in \Sigma}(-1)^{\codim(\sigma)} h^0_\sigma(u) =\sum_{\sigma\in\Sigma}\sum_{\tau \in \phi^{-1}(\sigma)}(-1)^{\codim(\tau)}h^0_\tau(u),$$
$$=\sum_{\tau\in \Sigma'}(-1)^{\codim(\tau)}h^0_\tau(u) =\chi(X_{\Sigma'},\E')_u.$$

\end{proof}

\begin{remark}
In \cite[Lemma 3.4]{CHK}, the above Theorem \ref{th-toric-pull-back} is proved for equivariant Euler characteristic of a toric vector bundle using algebra-geometric methods. Theorem \ref{th-toric-pull-back} extends this to tropical vector bundles.    
\end{remark}

\begin{theorem}\label{Euler-Ch}
        Let $\E$ be a tropical vector bundle over $X_\Sigma$ of rank $r$ with data $\Phi_\E : |\Sigma| \to \widetilde{\Berg}(\M)$. For the associated convex chain $\alpha_\E$,
        $$\chi(X_\Sigma,\E)_u=\alpha_\E(u).$$
\end{theorem}
\begin{proof}
    Fix $\sigma\in \Sigma$ and $u\in M$. We order the set $\{\h_i(\bv_\rho):=\displaystyle i^{th}\min_{1\leq j\leq r}\langle u_{\sigma,j},\bv_\rho\rangle\}$. Since we are taking minimums, $\h_i$ are not necessarily linear on cones of $\Sigma$ but rather 
    $\h_i$ are piecewise linear with respect to a refinement of the fan $\Sigma$. In view of Theorem \ref{th-toric-pull-back}, after replacing $\Sigma$ with a refinement of $\Sigma$, we can assume that the $\h_i$ are linear on each cone of $\Sigma$. Also, we note that by definition, $\alpha_\E$ is invariant under refining the fan.
    Let $$P_{i,\sigma} = \{u\in M_\R \mid \langle u,\bv_\rho\rangle \leq \h_i(\bv_\rho), \,\forall \rho\in \sigma(1)\rangle\}.$$
    By the definition of $h^0_\sigma(u)$ (see Remark \ref{h^0}), we have
    $$\chi(X_\Sigma,\E)_u = \sum_{\sigma\in \Sigma}(-1)^{\codim(\sigma)}h^0_\sigma(u)\nonumber  =\sum_{\sigma\in\Sigma}(-1)^{\codim(\sigma)}\left(\sum_{i=1}^r\mathbbm{1}_{P_{i,\sigma}}(u)\right).$$
        
Let $\alpha_{i}$ be the convex chain associated to $\h_i$. From the Brianchon-Gram for virtual polytope associated $\alpha_i$ (see Remark \ref{rem-Brianchon-Gram}) we get
    $$ \chi(X_\Sigma,\E)_u=\sum_{i=1}^r\left(\sum_{\sigma\in\Sigma} (-1)^{\codim(\sigma)}\mathbbm{1}_{P_{i,\sigma}}(u)\right)=\sum_{i=1}^r \alpha_i(u) = \alpha_\E(u).$$
    \end{proof}

Theorem \ref{Euler-Ch} now allows us to apply the Khovanskii-Pukhlikov Hirzebruch-Riemann-Roch (Theorem \ref{Thm-Kh-Pu}) to the convex chain $\alpha_\E$ to obtain a combinatorial Hirzebruch-Riemann-Roch for tropical vector bundles. 

For a vector
$z=(z_\rho)_{\rho\in\Sigma(1)} \in \mathbb{R}^{\Sigma(1)}$, let $P(z)$ be the virtual polytope with the support numbers $z_\rho$ (see paragraph after Theorem \ref{th-convex-chain-supp-function}).
For the convex chain $\alpha_\E$ associated to a tropical vector bundle $\E$ on $X_\Sigma$, define $\alpha_\E[z]:= \alpha_\E * \mathbbm{1}_{P
(z)}.$

\begin{corollary}[Hirzebruch-Riemann-Roch for tropical vector bundles]  \label{cor-HRR-trop-vb}
    Let $\E$ be a tropical vector bundle over $X_\Sigma$ and $\alpha_\E$ be the associated convex chain. Then,
    $$\Todd\left(\frac{\partial}{\partial z}\right) I(\alpha_\E[z]) \Big\vert_{z=0} =\chi(X_\Sigma,\E).$$

\end{corollary}

\begin{example}[Fano plane]
        We describe a tropical toric vector bundle $\E$ over $\P^2$ built from the Fano plane. (Figure \ref{fig-Fano}, also see \cite[Example 4.10]{KM-TVBs-valuations}).

\begin{figure}[ht]
\begin{tikzpicture}
\begin{scope}

\draw (30:2)  -- (210:4)
        (150:2) -- (330:4)
        (270:2) -- (90:4)
        (90:4)  -- (210:4) -- (330:4) -- cycle
        (0:0)   circle (2);

\filldraw[black] (0:0) circle (2pt) node[anchor=west]{$w$};

\filldraw[black] (30:2) circle (2pt) node[anchor=west]{$z_1$};
\filldraw[black] (150:2) circle (2pt) node[anchor=east]{$z_2$};
\filldraw[black] (270:2) circle (2pt) node[anchor=north]{$z_3$};

\filldraw[black] (90:4) circle (2pt) node[anchor=south]{$y_3$};
\filldraw[black] (210:4) circle (2pt) node[anchor=east]{$y_1$};
\filldraw[black] (330:4) circle (2pt) node[anchor=west]{$y_2$};

\end{scope}
\end{tikzpicture}
\caption{The Fano plane.} \label{fig-Fano}
\end{figure}
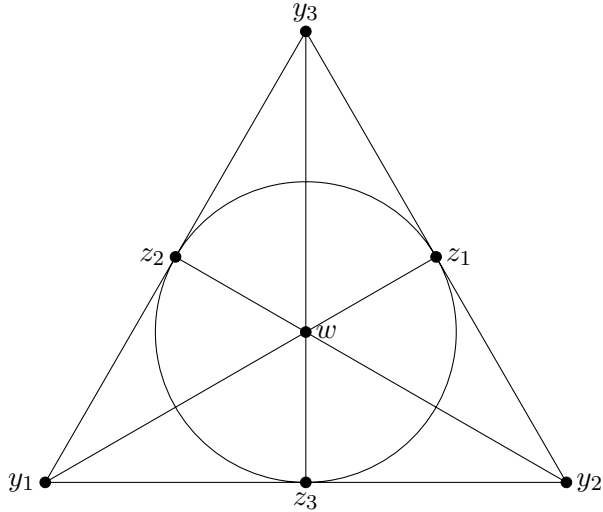
   
Let $\rho_1, \rho_2, \rho_3$ be the rays of the fan of $\P^2$ with ray generators $\{\bv_1,\bv_2,\bv_3\}$ where $\bv_1,\bv_2$ are the standard basis of $N$ and $\bv_3=-\bv_1-\bv_2$. To describe a tropical vector bundle over $\P^2$, it suffices to find three flags of flats in the Fano plane such that any pair of flags shares a common adapted basis.  We let decreasing filtrations by flats:
$$F^{\rho_i}_r =
\begin{cases}
\G & r \leq 0,\\
\{w,y_i,z_i\} & 0 < r \leq 1,\\
\{y_i\} & 1 < r \leq 2,\\
\emptyset & 2 < r.
\end{cases}$$
For $\rho_i$ and $\rho_j$ the filtrations share a common adapted basis
$B_{ij}=\{y_i,y_j,w\}$.
    This information is encoded in the following diagram:

\begin{center}
\begin{tabular}{ c|ccccccc } 
  & $y_1$ & $y_2$ & $y_3$ & $z_1$ & $z_2$ & $z_3$ & $w$\\ 
 \hline
 $\rho_1$ & 2 & 0 & 0 & 1 & 0 & 0 & 1\\ 
 $\rho_2$ & 0 & 2 & 0 & 0 & 1 & 0 & 1\\ 
 $\rho_3$ & 0 & 0 & 2 & 0 & 0 & 1 & 1\\ 
\end{tabular}
\end{center}
The parliament of polytopes are as in the Figure \ref{fig-PARLIAMENT}.

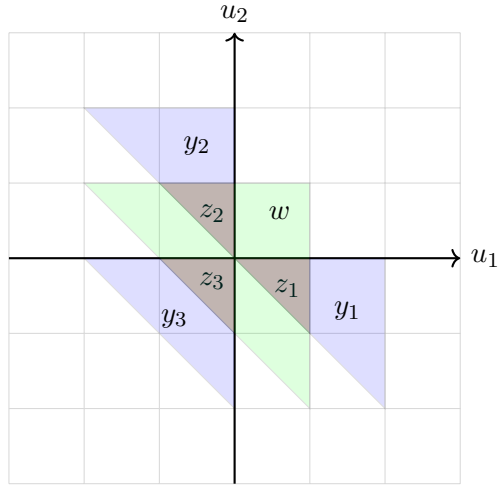
\begin{figure}[ht]
\begin{tikzpicture}
\draw[step=1, gray!30] (-3,-3) grid (3,3);
\draw[->, thick] (-3,0) -- (3,0) node[right] {$u_1$};
\draw[->, thick] (0,-3) -- (0,3) node[above] {$u_2$};

\draw[fill=red,opacity=0.23] (-1,0) -- (0,0) -- (0,-1) -- cycle;
\node at (-0.3,-0.3) {$z_3$};

\draw[fill=red,opacity=0.23] (0,1) -- (0,0) -- (-1,1) -- cycle;
\node at (-0.3,0.6) {$z_2$};

\draw[fill=red,opacity=0.23] (0,0) -- (1,0) -- (1,-1) -- cycle;
\node at (0.7,-0.4) {$z_1$};

\draw[fill=blue,opacity=0.13] (-2,0) -- (0,0) -- (0,-2) -- cycle;
\node at (-0.8,-0.8) {$y_3$};

\draw[fill=blue,opacity=0.13] (0,2) -- (0,0) -- (-2,2) -- cycle;
\node at (-0.5,1.5) {$y_2$};

\draw[fill=blue,opacity=0.13] (0,0) -- (2,0) -- (2,-2) -- cycle;
\node at (1.5,-0.7) {$y_1$};

\draw[fill=green,opacity=0.13] (-2,1) -- (1,1) -- (1,-2) -- cycle;
\node at (0.6,0.6) {$w$};

\end{tikzpicture}
\caption{The parliament of polytopes for $\E$.} \label{fig-PARLIAMENT}
\end{figure}

 We compute $\displaystyle\sum_{u\in M}h^0(u) = 27$ by the definition $h^0(u) = \rank \{e\in \G : u\in P_e\}$. \\

To compute $\alpha_\E$, let $\{\e_1,\e_2\}$ be the standard basis of $M$. For the full dimensional cones, $\sigma_1= \cone(\bv_2,\bv_3), \sigma_2=\cone(\bv_1,\bv_3),\sigma_3=\cone(\bv_1,\bv_2),$  the multi-sets of characters are
$u(\sigma_1)=\{-2\e_1+2\e_2,-2\e_1,-2\e_1+\e_2\}, u(\sigma_2)=\{2\e_1-2\e_2,-2\e_2,\e_1-2\e_2\}$, and $u(\sigma_3)=\{2\e_1,2\e_2,\e_1+\e_2\}$.
For $x=(x_1,x_2) \in N_\R$, 
$$\h_\E(x) = \begin{cases}
    [-2x_1+2x_2]+[-2x_1]+[-2x_1+x_2]& \text{if}\,\,x\in \sigma_1\\
    [2x_1-2x_2]+[-2x_2]+[x_1-2x_2]&\text{if}\,\, x\in \sigma_2\\
    [2x_1]+[2x_2]+[x_1+x_2]&\text{if}\,\, x\in \sigma_3.
\end{cases}$$
Consider the piecewise linear functions $\h_i$ for $1\leq i\leq 3$ as in the proof of Theorem \ref{Euler-Ch}. The virtual polytope $P_i$ associated to $\h_i$ is given by
$$P_i = \{u\in M_\R \mid \langle u,\bv_\rho\rangle\leq \h_i(\bv_\rho), \forall\rho\in \Sigma(1)\}.$$
Since $\h_2,\h_3$ are convex, $P_2,P_3$ are convex polytopes with $\displaystyle\sum_{u\in M} \mathbbm{1}_{P_2}=10$ and $\displaystyle\sum_{u\in M}\mathbbm{1}_{P_3}=19$.

The function $h_1$ is not convex. Write $\h_1=\h^+-\h^-$ with convex
support functions $\h^\pm$. In the convex chain algebra this gives
$$\mathbb{1}_{P_1}=\mathbb{1}_{P^+}*(\mathbb{1}_{P^-})^{-1},$$
and Minkowski inversion yields
$$\sum_{u\in M}\mathbb{1}_{P_1}(u)=-2.$$
Hence, $$\sum_{u\in M}\alpha_\E(u) =\sum_{u\in M} \mathbbm{1}_{P_1}(u)+\mathbbm{1}_{P_2}(u)+\mathbbm{1}_{P_3}(u)=27 =\sum_{u\in M}h^0(\E)_u$$
and the higher cohomologies vanish. 

\end{example}

\section{Global sections and Euler characteristic of the tautological bundle of a matroid}\label{sec-global sec-Euler c-taut}

We recall from \cite{KM-trop-vb} that each matroid $\M$ comes with a canonical tropical vector bundle $\E_\M$ on the permutahedral toric variety called \emph{tautological bundle on $\M$}. It was conjectured in \cite{KM-trop-vb} that $h^0(\E_\M)$ and $\chi(\E_\M)$ coincide. This was interpreted as \emph{higher cohomologies of $\E_\M$ vanish}. In this section, we give a proof of this.

We show that the higher cohomologies of this bundle vanish. Recall that the $m$-dimensional \emph{permutahedron} is the convex hull:
$$P_m = \conv\{\pi(1,\ldots,m)\mid \pi \in S_m\}.$$
The \emph{permutahedral fan}, $\Sigma_m$ is the normal fan of the permutahedron in $\R^m/\R\e_{[m]}.$ The corresponding toric variety $X_m$ is the \emph{permutahedral toric variety}. 

We define the \emph{lifted} permutahedral fan $\widetilde{\Sigma_m}$ to be the preimage of $\Sigma_m$ under the quotient map. The fan $\widetilde{\Sigma_m}$ encodes the action of $(\k^*)^m$ on $X_m$. 

\subsection{The map $\Phi_\M$ and the bundle $\E_\M$}
Let $w \in \R^m$. Consider the associated $\R$-filtration $(F_{w\geq k})_{k \in \R}$ by flats on $\G=[m]$ where
$$F_{w \geq k} = \Span\{ i \in [m] \mid w_i \geq k\}.$$
The filtration $(F_{w\geq k})_{k \in \R}$ determines a point $w' \in \widetilde{\Berg}(\M) \subset \R^m$ by:
$$w'_i = \sup\{k \mid i \in F_{w \geq k} \}.$$

Thus we obtain a canonical projection map:
$$\Phi_\M: |\widetilde{\GF}(\M)| \to \widetilde{\Berg}(\M), \quad \Phi_\M(w) = w'.$$

Moreover, by definition $\Phi_\M$ is the identity when restricted to $\widetilde{\Berg}(\M)$, so $\Phi_\M\circ \Phi_\M = \Phi_\M$. Now we can show that $\Phi_\M$ is in fact a piecewise linear map. For $\sigma \in \widetilde{\GF}(\M)$ consider the map $\Phi_\sigma: \sigma \to \widetilde{A}_{B_\sigma} = \sigma \cap \widetilde{\Berg}(\M)$ defined as follows. For $e_i \in \M$ let $C_i$ be the circuit in $\{e_i\} \cup B_\sigma$ containing $e_i$. Then for $w=(w_1, \ldots, w_m) \in \sigma$ we put:
$$\Phi_\sigma(w)_i = 
\begin{cases} 
w_i,  \quad e_i \in B \\
\min\{w_j \mid j \in C_i \setminus \{i\} \}, \quad i \notin B_\sigma.
\end{cases}$$

We recall that the $\widetilde{GF}(\M)$ is refined by the permutahedral fan $\widetilde{\Sigma_m}$. Thus we can consider $\Phi_\M$ as a piecewise linear map from $|\widetilde{\Sigma_m}|$ to $\widetilde{\Berg}(\M)$.

\begin{definition}[Tautological tropical vector bundle of a matroid]  \label{def-tautological-mb}
For a matroid $\M$, we call the tropical toric vector bundle $\E_\M$ given by the piecewise linear map $\Phi_\M: |\widetilde{\Sigma_m}| \to \widetilde{\Berg}(\M)$, the \emph{tautological tropical vector bundle of $\M$}. 
\end{definition}

We adopt the convention that $\widetilde{\GF}(\M)$ is the \emph{outer} normal fan of of the matroid polytope of $\M$.  In particular, the maximal faces of $\widetilde{\GF}(\M)$ are in bijection with bases $\B$, and the basis $\B_C$ of a maximal face $C \in \widetilde{\GF}(\M)$ is the set of tuples $w$ for which $\B$ is the basis of maximal total weight.  We let $\B_\sigma$ denote the basis associated to $\sigma \in \widetilde{\Sigma_m}$. Recall that $\e_i$ denotes the $i$-th standard basis vector of $\R^m$ and $\e_S = \sum_{i \in S} \e_i$ for $S \subset [m]$.

\begin{proposition}\label{pro-tautologicalGG}
    For $\sigma \in \widetilde{\Sigma_m}$, the adapted basis for $\Phi_\M$ restricted to $\sigma$ is $\B_\sigma$, and the character $u(\sigma)$ is $\{\e_i\mid i \in B_\sigma\}$. The diagram $D(\E_\M)$ is the matrix with $S, i$-the entry equal to $1$ if $i$ is in the span of the $j$ for $j \in S$, and $0$ otherwise.
\end{proposition}

\subsection{Vanishing of higher cohomologies of tautological bundle} \label{subsec-coho-tauto}

\begin{proposition}\label{prop-h0-tautological bundle}
$$h^0(X_m,\E_\M)_u= \begin{cases}
1 &  \text{if}\,\,u= \e_i \,\,\text{for a non-loop}\,\, i \in E\\
0 & otherwise.
\end{cases}$$
    
\end{proposition}
\begin{proof}
    For a loop, $i \in E$, by Definition \ref{def-parliament}, we have $P_i =\{u \in \Z^m \mid \langle u,\e_S\rangle\leq (\e_{\Span(S)})_i, \forall S \subseteq \G\}$. Choose $u \in \Z^m$ with $\sum u_j = 1$. Then,
    $u \in P_i$ if and only if $u = \e_j$ for some $j \in \G$. This shows that $P_i \cap \{u\in \Z^m\mid\sum u_k = 1\} = \{\e_j \in \Z^m \mid j \in \G\}$. 

    For a non-loop, $i \in \G$, let $C(i) := \{ j \in \G\mid j \,\textup{ parallel to }\, i\}$. Assume $u \in  P_i$ with $\sum u_j =1$. If $j\notin C(i)$, then $u_j \leq 0$ by taking $S= \{j\}$ in the definition of $P_i$. Also, taking $S = \G\setminus \{j\}$, we have $\displaystyle\sum_{k \neq j} u_k \leq 1$. Therefore, $u_j =0$ for all $j \notin C(i)$ and $\displaystyle\sum_{k\in C(i)} u_k =1$. Repeating the argument as in the case of a loop, we conclude $u\in P_i \cap \{u\in \Z^m \mid \sum u_k = 1\}$ if and only if $u = \e_j $ for some $j \in C(i)$. Hence,
$$h^0(X_m,\E_\M)_u= \begin{cases}
\rank(C(i)\cup \{\text{loops}\})=1 &  \text{if}\,\,u= \e_i \,\,\text{for a non-loop}\,\, i \in \G\\
\rank(\{\text{loops}\})=0 & \text{otherwise.}
\end{cases}$$
\end{proof}

    \begin{proposition}\label{prop-weighted-h0}
    Let $\sigma \in \widetilde{\Sigma_m}$ be a cone with its flag of subsets of $\G$, $\emptyset =S_0 \subsetneq S_1 \subsetneq \ldots \subsetneq S_{l(\sigma)}=\G$. Let $S_t$ be the first place where $\langle u ,{\bf e}_{S_t}\rangle = 1$. Then,
        $$h^0(U_\sigma,\E_\M)_u =\begin{cases}
            0 \quad\textup{if } \exists k, \langle u,{\bf e}_{S_k}\rangle>1\\
            \rank(S_t) \quad \textup{otherwise}.
        \end{cases}$$
    \end{proposition}
\begin{proof}
    We fix a character $u$ with $\sum_j u_j=1.$ and assume that $\langle u, \e_{S_k}\rangle \leq 1$ for all $k$. Let $i \in H^0(U_\sigma,\E_\M)_u=\{j \in \G \mid u \in P_{\sigma,j}\}.$ Then, $1= \langle u ,\e_{S_t} \rangle \leq (\e_{\Span(S_t)})_i$. By taking $k=t$, we conclude $i \in \Span(S_t)$. Conversely, if $i \in \Span(S_t)$, then $\langle u, \e_{S_k}\rangle \leq (\e_{\Span(S_k)})_i$ for all $k$ by the minimality of $t$.

\end{proof}
Now, we give a description of $u$-weighted Euler characteristic of the tautological bundle. For a character $u\in \Z^n $ with $\sum_j u_j = 1$ and a non-empty subset $S\subseteq \G$, we define a collection of flags of subsets of $\G$ denoted by $\Pi_u(S)$ as follows:
$$\Pi_u(S) = \{\pi \mid \exists t, S_t=S\,\text{appearing in the flag $\pi$ as the first place}\, \langle u,\e_{S_t}\rangle =1,\,\text{and} \langle u,\e_{S_j}\rangle \leq1 ,\,\forall j\}$$ and $\Pi_u(S) =\emptyset$ if $\langle u, \e_S\rangle \neq 1$ for every flat $S$ appearing in $\pi$.
\begin{corollary}\label{cor-weighted euler char}
\begin{equation}\label{eq-weighted-euler-char}
\chi(X_{m},\E_\M)_u=  \displaystyle\sum_{\sigma\in \Sigma_m} (-1)^{\codim(\sigma)}h^0(U_\sigma,\E_\M)_u = \sum_{\substack{S \subseteq \G \\ \langle u, \e_S \rangle = 1}} \sum_{\pi \in \Pi_u(S)} (-1)^{n-l(\pi)}\rank(S).    
\end{equation}

\end{corollary}

\begin{lemma} \label{lemma-mobius function of poset of flags}
    Let $\Pi^m$ be the poset of all flags of subsets of $[m]$. Then, $\displaystyle\sum_{\pi\in \Pi^m} (-1)^{l(\pi)}=(-1)^m.$
\end{lemma}
\begin{proof}
    This is the Brianchon-Gram formula for the permutahedral fan.
\end{proof}
\begin{theorem}[Vanishing of higher cohomology of tautological bundle]\label{thm-vasnishing of higher coh-taut}
    $$\chi(X_{m},\E_\M)_u= h^0(X_{m},\E_\M)_u.$$
\end{theorem}
\begin{proof}
We prove the identity by showing the coefficient of $\rank(S)$ in Corollary \ref{cor-weighted euler char} is $\rank(i)$ if $u=\e_i$ and $S =\{i\}$ and $0$ otherwise.

Fix $u = \e_i$ and let $S=\{i\}$, then $\Pi_u(S)\simeq \Pi^{m-1}$. By Lemma $\ref{lemma-mobius function of poset of flags}$,
$$\sum_{\pi \in \Pi_u(S)} (-1)^{m-l(\pi)} = \sum_{\pi'\in \Pi^{m-1}}(-1)^{m-1-l(\pi')} =1.$$ 
Let $u=\e_i$, and $ S \neq \{i\}$ with $\langle u,\e_S\rangle=1$, for $\pi \in \Pi_u(S),$ define an involution $\pi^*$ by removing or adding $S\setminus\{i\}$ in the flag, $\pi$. Then, the alternating sum $\displaystyle\sum_{\pi\in \Pi_u(S)} (-1)^{m-l(\pi)}=0$.

Similarly, for other $u$ and $S\subseteq \G$, we define an involution, $\pi^*$ to complete the proof. For a fixed $u\in \Z^m$ with $\sum_j u_j=1$, let $a \in \G$ be  a coordinate such that $u_a<0.$ Let $\pi \in \Pi_u(S)$ with  $\pi: S_1\subsetneq \ldots \subsetneq S \subsetneq \ldots \subsetneq \G.$ Define $\pi^*$ as follows: for $k$ with $a \in S_k\setminus S_{k-1},$
$$\text{if }|S_{k-1}\setminus S_k|>1,\,\text{replace } S_k \,\text{with } S_{k-1}\cup \{a\} \subsetneq S_k, \,\text{and }$$
$$\text{if }|S_{k-1}\setminus S_k|=\{a\},\,\text{replace } S_k \subsetneq S_{k+1} \,\text{with } S_{k+1}.$$
\end{proof} 

\begin{remark}
The vanishing of higher cohomologies for the tautological bundle is consistent with result in the representable case. 
In particular, this agrees with a main theorem of \cite{Eur} for tautological bundles of representable matroids. \end{remark}

\begin{example}[Uniform matroid $U_{2,3}$]
We compute $h^0(X_{m},\E_\M)$ and $\chi(X_m,\E_\M)$ for a tropical vector bundle, $\E$ that comes from the uniform matroid $\M = U_{2,3}$ on $\G=\{1,2,3\}$. 
Let $\widetilde{\Sigma_3}\subset \R^3$ be the lifted permutahedral fan. The rays are indexed by the nonempty proper subsets
$S\subset E$, with primitive generators $e_S=\sum_{i\in S} e_i$.
The Klyachko data of the tautological bundle $\E_{U_{2,3}}$ is given by the diagram
$$
\begin{array}{c|ccc}
S & 1 & 2 & 3 \\ \hline
\{1\}   & 1 & 0 & 0 \\
\{2\}   & 0 & 1 & 0 \\
\{3\}   & 0 & 0 & 1 \\
\{1,2\} & 1 & 1 & 0 \\
\{1,3\} & 1 & 0 & 1 \\
\{2,3\} & 0 & 1 & 1
\end{array}
$$
The parliament of polytopes is determined by the columns of the diagram on the affine hyperplane $\sum_{j}u_j=1$. Therefore, $P_i= \{e_i\}.$ for all $1\leq i \leq3$. It follows that for all $u\in M$ with $\sum_{j}u_j=1$,
$$h^0(X_3,\E_{U_{2,3}})_u =\begin{cases}
1 & \text{if } u\in\{e_1,e_2,e_3\},\\
0 & \text{otherwise.}
\end{cases}$$
For $u=e_1$, maximal cones correspond to orders $(a,b,c)$ of $\{1,2,3\}$ with flag $\{a\}\subset\{a,b\}\subset E$, and
$h^0(U_\sigma,\E_{U_{2,3}})_{e_1}=1$ if $a=1$ and $2$ otherwise, giving total $10$ over the maximal cones. 
Rays correspond to nonempty proper $S\subset E$, and $h^0(U_\rho,\E_{U_{2,3}})_{e_1}=1$ for $S=\{1\}$ and $2$ for the other rays, giving total $11$ over the rays. For the zero cone the value is $2$.
Hence,
$$\chi(X_{3},\E_{U_{2,3}})_{e_1}=10-11+2=1,$$ and by symmetry $\chi(X_{3},\E_{U_{2,3}})_{e_i}=1$ for $i=1,2,3$. For other $u$ with $\sum_{j}u_j=1,$ the involution on flags used in the proof of Theorem \ref{thm-vasnishing of higher coh-taut} cancels the alternating sum, hence $\chi(X_{3},\E_{U_{2,3}})_u=0.$

\end{example}



\end{document}